\documentclass{amsart}

\usepackage{amssymb}
\usepackage{graphicx}
\usepackage{MnSymbol}

\usepackage{amsmath,amscd}

\usepackage{amsthm}

\title[$\aleph_1$-free groups]{Automatic continuity of $\aleph_1$-free groups}

\theoremstyle{definition}\newtheorem{theorem}{Theorem}
\theoremstyle{definition}
\theoremstyle{definition}
\theoremstyle{definition}
\theoremstyle{definition}

\theoremstyle{definition}\newtheorem{bigtheorem}{Theorem}

\theoremstyle{definition}
\theoremstyle{definition}
\theoremstyle{definition}\newtheorem{definition}[theorem]{Definition}
\theoremstyle{definition}
\theoremstyle{definition}
\theoremstyle{definition}
\theoremstyle{definition}
\theoremstyle{definition}\newtheorem{lemma}[theorem]{Lemma}
\theoremstyle{definition}
\theoremstyle{definition}
\theoremstyle{definition}
\theoremstyle{definition}

\newcommand{\W}{\mathcal{W}}
\newcommand{\HEG}{\operatorname{HEG}}

\begin{document}

\author[Samuel M. Corson]{Samuel M. Corson}
\address{Ikerbasque- Basque Foundation for Science and Matematika Saila, UPV/EHU, Sarriena S/N, 48940, Leioa - Bizkaia, Spain}
\email{sammyc973@gmail.com}

\thanks{The author is
supported by European Research Council grant PCG-336983.}

\begin{abstract}  We prove that groups for which every countable subgroup is free ($\aleph_1$-free groups) are n-slender, cm-slender, and lcH-slender.  In particular every homomorphism from a completely metrizable group to an $\aleph_1$-free group has an open kernel.  We also show that $\aleph_1$-free abelian groups are lcH-slender, which is especially interesting in light of the fact that some $\aleph_1$-free abelian groups are neither n- nor cm-slender.  The strongly $\aleph_1$-free abelian groups are shown to be n-, cm-, and lcH-slender.  We also give a characterization of cm- and lcH-slender abelian groups.
\end{abstract}

\maketitle

\begin{section}{Introduction}
Graham Higman defined a group to be \emph{$\kappa$-free}, with $\kappa$ a cardinal number, if each subgroup generated by fewer than $\kappa$ elements is a free group \cite{H1}.  By the Nielsen-Schreier Theorem each free group is $\kappa$-free for all cardinals $\kappa$.  The additive group of the rationals $\mathbb{Q}$ is an example of an $\aleph_0$-free group of cardinality $\aleph_0$ which is not free.  Higman produced an example of an $\aleph_1$-free group of cardinality $\aleph_1$ which is not free, and $\kappa$-free groups have been a focus of much study since then (\cite{H1}, \cite{S}, \cite{EkMe}, \cite{MaS}).  We prove that $\aleph_1$-free groups satisfy strong automatic continuity conditions.

Following \cite{CC} we define a group $H$ to be \emph{cm-slender} if every abstract homomorphism from a completely metrizable topological group to $H$ has open kernel.  Similarly $H$ is \emph{lcH-slender} provided each abstract homomorphism from a locally compact Hausdorff topological group to $H$ has open kernel.  If, for example, a group $H$ is cm-slender then the only completely metrizable topology that can be imposed on $H$ to make $H$ a topological group is the discrete topology.

A further notion of automatic continuity comes from fundamental groups: a group $H$ is \emph{n-slender} if every abstract group homomorphism from the fundamental group $\HEG$ of the Hawaiian earring to $H$ factors through a finite bouquet of circles \cite{Ed} (see Section \ref{Automatic}).  Free (abelian) groups were shown to be cm- and lcH-slender in \cite{D} and free groups were shown to be n-slender in \cite{H2}.  Many groups have since been shown to be n-, cm- and lcH-slender, and each of these notions of slenderness requires a group to be torsion-free and to not have $\mathbb{Q}$ as a subgroup (see \cite{CC} for more exposition).

We prove the following:

\begin{bigtheorem}\label{thebigone} $\aleph_1$-free groups are n-slender, cm-slender, and lcH-slender.
\end{bigtheorem}

As free groups are $\aleph_1$-free, this result is a strengthening of the classical facts that free groups are n-, cm- and lcH-slender.  The fact that $\aleph_1$-free groups are cm-slender immediately implies a result of Khelif \cite{Kh} that an uncountable $\aleph_1$-free group is not the homomorphic image of a \emph{Polish group} (a topological group which is separable and completely metrizable).  The cm-slenderness of $\aleph_1$-free groups can be obtained by modifying Khelif's proof.  We give a different proof which is both well suited to proving all three types of slenderness and seemingly simpler.  

Theorem \ref{thebigone} cannot be strengthened by substituting $\aleph_0$-freeness (that is, local freeness) for $\aleph_1$-freeness.  The group $\HEG$ is itself locally free and by considering the identity map we see that local freeness does not imply that a group is n-slender.  The group $\mathbb{Q}$ is locally free, and using a Hamel basis of $\mathbb{R}$ over $\mathbb{Q}$ it is possible to construct a homomorphism from $\mathbb{R}$ to $\mathbb{Q}$ which is not continuous.  Since $\mathbb{R}$ is both locally compact Hausdorff and completely metrizable, local freeness implies neither cm- nor lcH-slenderness.

Analogously define a group to be \emph{$\kappa$-free abelian} if each subgroup generated by fewer than $\kappa$ elements is free abelian.  A group which is $\aleph_1$-free abelian needn't be n- or cm-slender: the countably infinite product $\prod_{\omega}\mathbb{Z}$ is $\aleph_1$-free abelian \cite{B}.  This group has a completely metrizable topological group structure given by taking each $\mathbb{Z}$ to be discrete and giving the entire group the product topology, and so the identity map on $\prod_{\omega}\mathbb{Z}$ shows that an $\aleph_1$-free abelian group need not be cm-slender.  Also there is a canonical homomorphism from $\HEG$ to $\prod_{\omega}\mathbb{Z}$ which does not have open kernel, so n-slenderness needn't hold for an $\aleph_1$-free abelian group either.  However we have the following:

\begin{bigtheorem}\label{lcHslender} $\aleph_1$-free abelian groups are lcH-slender.
\end{bigtheorem}

Theorem \ref{lcHslender} cannot be strengthened by replacing $\aleph_1$ with $\aleph_0$ since $\mathbb{Q}$ is not lcH-slender.   We prove Theorem \ref{lcHslender} from the following classification (see definitions in Section \ref{abeliancase}):

\begin{bigtheorem}\label{characterizationoflcHslenderabelian}  If $H$ is an abelian group then
\begin{enumerate}

\item $H$ is cm-slender if and only if $H$ is torsion-free, reduced and contains no subgroup which admits a non-discrete Polish topology

\item $H$ is lcH-slender if and only if $H$ is cotorsion-free

\end{enumerate}
\end{bigtheorem}

The n-slender abelian groups are already known to be precisely the slender groups \cite{Ed}, and Theorem \ref{characterizationoflcHslenderabelian} was already known for abelian groups of cardinality $<2^{\aleph_0}$ (see \cite[Theorem C]{CC}).  Thus among abelian groups we have

\begin{center}
 cm-slender $\Longrightarrow$ n-slender $\Longrightarrow$ lcH-slender
\end{center}

For n- and cm-slenderness we need to demand a bit more from an $\aleph_1$-free abelian group (see Definition \ref{stronglydef}):

\begin{bigtheorem}\label{strongly}  Strongly $\aleph_1$-free abelian groups are n-slender, cm-slender and lcH-slender.
\end{bigtheorem}

We have already seen that the modifier ``strongly'' may not be dropped while concluding n- and cm-slenderness.  In Section \ref{Automatic} we prove Theorem \ref{thebigone} and in Section \ref{abeliancase} we prove Theorems \ref{characterizationoflcHslenderabelian}, \ref{lcHslender} and \ref{strongly}.
\end{section}

\begin{section}{Automatic continuity in the non-abelian case}\label{Automatic}

We begin this section with a review of the Hawaiian earring group $\HEG$.  After this we give background lemmas and prove Theorem \ref{thebigone}.  Start with a countably infinite set $\{a_n^{\pm 1}\}_{n\in \omega}$ which has formal inverses.  We say a function $W: \overline{W} \rightarrow \{a_n^{\pm 1}\}_{n\in \omega}$ is a \emph{word} if the domain $\overline{W}$ is a totally ordered set and for each $m$ the preimage $W^{-1}(\{a_{n}^{\pm 1}\}_{n=0}^m)$ is finite.  We write $W \equiv U$ for words $W$ and $U$ provided there exists an order isomorphism $\iota: \overline{W} \rightarrow \overline{U}$ such that $W(i) = U(\iota(i))$.  Let $\W$ denote a selection from each $\equiv$ class.  For $m\in \omega$ let $p_m$ denote the map from $\W$ to the set of finite words given by the restriction $p_m(W) \equiv W\upharpoonright\{i\in \overline{W}\mid W(i)\in \{a_{n}^{\pm 1}\}_{n=0}^m\}$.  

For $W, U \in \W$ we write $W \sim U$ if for every $m\in \omega$ the words $p_m(W)$ and $p_m(U)$ are equal as elements in the free group $F(a_0, \ldots, a_m)$.  For $U \in \W$ we write $U^{-1}$ for the word whose domain is $\overline{U}$ under the reverse order satisfying $U^{-1}(i) = (U(i))^{-1}$.  We concatenate two words $W, U\in \W$ by letting $\overline{WU}$ be the disjoint union $\overline{W} \sqcup \overline{U}$ under the order which preserves that of both $\overline{W}$ and $\overline{U}$ and places elements in $\overline{W}$ below those of $\overline{U}$.  The map $WU$ is given by $WU(i) = \begin{cases}W(i)$ if $i\in \overline{W}\\ U(i)$ if $i\in \overline{U}\end{cases}$

The quotient set $\HEG = \W/\sim$ has a group structure given by $[W][U] = [WU]$ and $[U]^{-1} = [U^{-1}]$.  The free group $F(a_0, \ldots, a_m)$ embeds naturally into $\HEG$ by considering finite words in $\{a_n\}_{n=0}^m$ as words as defined above, and we let $\HEG_m$ denote this copy of the free group.  Each aforementioned map $p_m: \W \rightarrow \W$ induces a homomorphic retraction $p_m:\HEG \rightarrow \HEG_m$.  For each $m$ we similarly have a word map $p^m(W) \equiv W\upharpoonright\{i\in \overline{W}\mid W(i)\notin \{a_{n}^{\pm 1}\}_{n=0}^m\}$ which defines a retraction to the subgroup $\HEG^m$ consisting of those elements of $\HEG$ which have a representative $W$ for which $W(\overline{W}) \cap \{a_{n}^{\pm 1}\}_{n=0}^m = \emptyset$.  There is a natural decomposition $\HEG \simeq \HEG_m *
\HEG^m$ for each $m$ given by considering a word as a finite concatenation of words utilizing elements in $\{a_n^{\pm 1}\}_{n=0}^m$ and words which do not.  The following definition is found in \cite{Ed}:

\begin{definition}\label{nslender}  A group $H$ is \emph{n-slender} if for every homomorphism $\phi: \HEG \rightarrow H$ there exists $m\in \omega$ for which $\phi  = \phi \circ p_m$.  Equivalently $H$ is n-slender if for every homomorphism $\phi: \HEG \rightarrow H$ there exists $m\in \omega$ such that $\HEG^m \leq \ker(\phi)$.
\end{definition}

We will make use of the following (see \cite[Theorem 1]{H1}):

\begin{lemma}\label{basic}  If $H$ is an $\aleph_1$-free group then each nondecreasing sequence $\{K_n\}_{n\in \omega}$ of finitely generated subgroups of $H$ such that $K_n$ is not contained in a proper free factor of $K_{n+1}$ must eventually stabilize.  Moreover every finitely generated $H_0 \leq H$ is included in a finitely generated $H_0 \leq H_1$ such that $H_1$ is a free factor of each free subgroup of $H$ which contains it.
\end{lemma}

We call such a subgroup $H_1$ as is asserted in the second sentence of Lemma \ref{basic} a \emph{basic} subgroup \cite{H1}.

\begin{lemma}\label{nicelemma} The following hold:

 \begin{enumerate} \item  If $\phi: \HEG \rightarrow H$ is a homomorphism to an $\aleph_1$-free group then for every finitely generated $F \leq H$ there exists $n\in \omega$ for which $\phi(\HEG^n) \cap F = \{1_H\}$.

\item  If $\phi: G \rightarrow H$ is a homomorphism with $G$ either completely metrizable or locally compact Hausdorff and $H$ an $\aleph_1$-free group then for every finitely generated $F \leq H$ there exists an open neighborhood $U \subseteq G$ of $1_G$ such that $\phi(U) \cap F = \{1_H\}$.

\end{enumerate}

\end{lemma}

\begin{proof} (1)  Assume the hypotheses and let $F\leq H$ be a finitely generated free subgroup.  By Lemma \ref{basic} we can select a finitely generated basic subgroup $F \leq H_1 \leq H$.  Fix a free generating set for $H_1$ and let $L: H_1 \rightarrow \omega$ be the associated length function.

Suppose for contradiction that $\phi(\HEG^n)\cap F$ is nontrivial for all $n$.  For each $n\in \omega$ select $W_n \in \HEG^n \setminus \ker(\phi)$.  Let $h_n = \phi(W_n)$  and let $k_n = L(\phi(W_n))$.  Let $\{U_n\}_{n\in \omega}$ be the sequence of words such that $U_n = W_n(U_{n+1})^{k_n +2}$.  Intuitively we have $U_0 = W_0(W_1(\cdots)^{k_1+2})^{k_0+2}$.  Let $z_n = \phi(U_n)$ for all $n\in \omega$.  Let $H_2 = \langle H_1 \cup \{z_n\}_{n\in \omega}\rangle$.  Since $H_2$ is a countable subgroup of $H$ we know $H_2$ is free and therefore $H_1$ is a free factor.  Let $\rho: H_2 \rightarrow H_1$ be any retraction induced by selecting a complimentary free factor and projecting to $H_1$.   Letting $y_n = \rho(z_n)$ we obtain the relations

\begin{center}
$y_n = h_n(y_{n+1})^{k_n+2}$
\end{center}

\noindent If $y_n \neq 1_H$ then 

\begin{center}
$L(y_{n-1}) \geq L(y_n^{k_n+2}) - L(h_{n-1})$

$\geq L(y_n) + k_n +1 - L(h_{n-1})$

$= L(y_n) +1$
\end{center}

\noindent and so $y_{n-1}\neq 1_H$ and $L(y_{n-1}) \geq L(y_n) +1$ and arguing backwards we see that for $m\geq n$ if $y_m \neq 1_H$ then $y_n \neq 1_H$ and $L(y_n)\geq L(y_m) + (m-n)$.  This implies that the $y_n$ are eventually trivial.  But then for some $n$ we have $y_n = 1_H = y_{n+1}$, from which we have $\phi(W_n) = h_n = y_ny_{n+1}^{-k_n-2} = 1_H$, contrary to the choosing of $W_n \notin \ker(\phi)$.

(2)  Suppose first that $\phi: G \rightarrow H$ is a homomorphism from a completely metrizable group to an $\aleph_1$-free group and that $F \leq H$ is finitely generated.  Let $d$ be a complete metric for $G$ compatible with the topology.  Select a finitely generated basic subgroup $H_1 \geq F$ and let $L$ be the length function for a fixed free generating set on $H_1$.  If a neighborhood $U$ as in the conclusion does not exist then we select $g_0 \in \phi^{-1}(F \setminus \{1_H\})$.  Let $k_0 = L(\phi(g_0))$.  Select a neighborhood $U_1$ of $1_G$ sufficiently small that $g\in U_1$ implies

\begin{center}  $d(g_0g^{k_0+2}, g_0) \leq \frac{1}{2}$

$d(g, 1_G)\leq \frac{1}{2}$

\end{center}

\noindent Select $g_1\in U_1 \cap \phi^{-1}(F \setminus \{1_G\})$ and let $k_1 = L(\phi(g_1))$.  Supposing that we have selected group elements $g_0, \ldots, g_n$ and neighborhoods $U_1, \ldots, U_n$ and natural numbers $k_0, \ldots, k_n$ in this way we select a neighborhood $U_{n+1}$ of $1_G$ for which $g\in U_{n+1}$ implies

\begin{center}  $d(g_0(\cdots g_{n-1}(g_n(g)^{k_n+2})^{k_{n-1}+2} \cdots)^{k_0+2}, g_0(g_1(\cdots g_{n-1}(g_n)^{k_{n-1}+2} \cdots)^{k_1+2})^{k_0+2})\leq\frac{1}{2^{n+1}}$

$d(g_1(\cdots g_{n-1}(g_n(g)^{k_n +2})^{k_{n-1}+2} \cdots)^{k_1+2}, g_1(g_2(\cdots g_{n-1}(g_n)^{k_{n-1}+2} \cdots)^{k_2+2})^{k_1+2})\leq \frac{1}{2^{n+1}}$

$\vdots$

$d(g_{n-1}(g_n(g)^{k_n+2})^{k_{n-1}+2}, g_{n-1}(g_n)^{k_{n-1}+2})\leq\frac{1}{2^{n+1}}$

$d(g_n(g)^{k_n+2}, g_n)\leq\frac{1}{2^{n+1}}$

$d(g, 1_G)\leq \frac{1}{2^{n+1}}$
\end{center}

\noindent Select $g_{n+1}\in U_{n+1}\cap \phi^{-1}(F\setminus \{1_G\})$ and let $k_{n+1} = L(\phi(g_{n+1}))$.  For each $n\in \omega$ the sequence $g_n(\cdots g_{m-1}(g_m)^{k_{m-1} +2} \cdots)^{k_n+2}$ is Cauchy in $m$ and therefore converges to some $j_n = \lim_{m \rightarrow \infty}g_n(\cdots g_{m-1}(g_m)^{k_{m-1} +2} \cdots)^{k_n+2}$ and by continuity of multiplication we have $j_n = g_{n+1}j_{n+1}^{k_{n+1} + 2}$.  Let $z_n = \phi(j_n)$ for all $n\in \omega$.  By again letting $H_2 = \langle H_1 \cup \{z_n\}_{n\in \omega}\rangle$ and $\rho$ being any retraction from $H_2$ to $H_1$ we obtain a contradiction as in part (1).

Suppose now that $\phi: G \rightarrow H$ is a homomorphism with locally compact Hausdorff domain and $\aleph_1$-free image and that for some finitely generated $F \leq H$ there is no $U$ as in the conclusion.  Select $H_1 \leq H$ as in the other case and again let $L$ be the length function with respect to a fixed free generating set for $H_1$.  Let $U_0$ be an open neighborhood of $1_G$ for which $\overline{U_0}$ is compact.  Select $g_0 \in U_0 \cap \phi^{-1}(F \setminus \{1_G\})$.  Let $k_0 = L(\phi(g_0))$.  Supposing we have selected elements $g_0, \ldots, g_n$ and nesting neighborhoods $U_0, \ldots, U_n$ of $1_G$ and natural numbers $k_0, \ldots, k_n$ in this way, we select a neighborhood $U_{n+1}\subseteq U_n$ of $1_G$ such that $g\in U_{n+1}$ implies $g_ng^{k_n+2}\in U_n$.  Let $g_{n+1}\in U_{n+1} \cap \phi^{-1}(F\setminus \{1_G\})$ and let $k_{n+1} = L(\phi(g_{n+1}))$.

For each $n\in \omega$ we let $K_n = g_0(g_1(\cdots g_n(\overline{U_{n+1}})^{k_n +2} \cdots)^{k_1 +2})^{k_0+2}$.  The sequence $\{K_n\}_{n\in \omega}$ consists of nonempty nesting compacta and so the intersection is nonempty.  Let $j_0 \in \bigcap_{n\in \omega} K_n$ and for each $n\geq 1$ we select $j_n\in \overline{U_n}$ such that $$j_0 = g_0(g_1(\cdots g_{n-1}j_n^{k_{n-1}+2} \cdots)^{k_1 +2})^{k_0 +2}$$  Let $z_n = \phi(j_n)$ for each $n\in \omega$.  Since $H$ is locally free we notice that $z_n = \phi(g_n)z_{n+1}^{k_n +2}$ for all $n$.  We argue as before for a contradiction.
\end{proof}

\begin{proof}(of Theorem \ref{thebigone})  We prove n-slenderness first and the arguments of the other types of slenderness will follow the same format.  Suppose $\phi: \HEG \rightarrow H$ is a map with $\aleph_1$-free codomain and imagine for contradiction that $\phi(\HEG^n)$ is never trivial.  Select $W_0\in \HEG \setminus \ker(\phi)$.  We have $\langle \phi(W_0)\rangle$ contained in a finitely generated basic free subgroup $F_0 \leq H$.  By Lemma \ref{nicelemma} pick $m_1 \in \omega$ large enough that $\phi(\HEG^{m_1}) \cap F_0$ is trivial.  Select $W_1\in \HEG^{m_1} \setminus \ker(\phi)$.  The finitely generated subgroup $\langle F_0 \cup \phi(W_1)\rangle$ is contained in a finitely generated basic subgroup $F_1$.  Supposing we have selected group elements $W_0, \ldots, W_n$ and basic subgroups $F_0\leq  \ldots \leq  F_n$ and natural numbers $m_0< \ldots <m_n$ in this way we select $m_{n+1}>m_n$ for which $\phi(\HEG^{m_{n+1}}) \cap F_n = \{1_H\}$.  Pick $W_{n+1}\in \HEG^{m_{n+1}}\setminus \ker(\phi)$ and let $F_{n+1}$ be a finitely generated basic subgroup which includes $\langle F_n \cup \{\phi(W_{n+1})\}\rangle$.

Define words $U_0, U_1, \ldots$ by $U_n = W_n^2U_{n+1}^2$.  Let $h_n = \phi(W_n)$ and $y_n = \phi(U_n)$ for all $n\in \omega$.  We consider the subgroup $H_{\infty} = \langle \{h_n\}_{n\in \omega}\cup \{y_n\}_{n\in \omega}\rangle \leq H$.

Notice first that for each $n\in \omega$ the elements $h_0, \ldots, h_n$ freely generate a subgroup of $H$.  This claim is obvious for $n = 0$.  Supposing the claim is true for $n$ we have $\langle h_0, \ldots, h_n\rangle\leq F_n$ and since $h_{n+1} = \phi(W_{n+1})\notin F_n$ we see that $F_n$ is a proper free factor of the group $\langle F_n \cup \{h_{n+1}\}\rangle$.  Since finitely generated free groups are Hopfian we know that if we fix a free generating set $X_n$ for $F_n$, the elements $X_n \cup \{h_{n+1}\}$ freely generate a subgroup of $H$.

Next, we claim the elements $h_0, \ldots, h_n, y_{n+1}$ freely generate a subgroup of $H$.  We have already seen that $h_0, \ldots, h_n$ freely generate a subgroup of the group $F_n$.  If $y_{n+1}$ is nontrivial then since evidently $y_{n+1} \in \phi(\HEG^{m_{n+1}})$ we can argue as before that $h_0, \ldots, h_n, y_{n+1}$ freely generates a subgroup.  Were $y_{n+1} = h_{n+1}^2y_{n+2}^2$ trivial, we would have $h_{n+1} = y_{n+2}$ since $H$ is locally free.  Then $h_{n+1} = y_{n+2} \in F_{n+1}\cap \phi(\HEG^{m_{n+2}}) = \{1_H\}$, contrary to how $W_{n+1}$ was chosen.

Letting $H_n = \langle h_0, \ldots, h_n, y_{n+1} \rangle$ it is easy to see that each $H_n$ is properly contained in $H_{n+1}$ and is not a free factor (one can use \cite[Lemma 7]{H1}, for example).  This contradicts Lemma \ref{basic}.  This group $H_{\infty} = \bigcup_{n\in \omega}H_n$ was identified by Higman as being a subgroup of $\HEG$ (see the discussion following \cite[Theorem 6]{H2}).

Suppose now that $\phi: G \rightarrow H$ is a homomorphism from a completely metrizable group to an $\aleph_1$-free group.  Suppose $\ker(\phi)$ is not open.  Select $g_0 \in G \setminus \ker(\phi)$.  Pick a finitely generated basic subgroup $F_0$ for which $\phi(g_0) \in F_0$.  By Lemma \ref{nicelemma} select an $\epsilon_1>0$ such that for $g$ in the open ball $B(1_G, \epsilon_1)$ we have

\begin{center}  $d(g_0^2(g)^2, g_0^2) \leq \frac{1}{3}$
\end{center}

\noindent  and for $g\in B(1_G, \epsilon_1) \setminus \ker(\phi)$ that $\phi(g) \notin F_0$.  Select $g_1$ such that $g_1, g_1^2 \in B(1_G, \frac{\epsilon_1}{3})\setminus \ker(\phi)$.  Select a finitely generated basic free subgroup $F_2$ of $H$ for which $F_1 \geq \langle F_0 \cup \{\phi(g_1)\}\rangle$.  Select $\epsilon_2>0$ such that $g\in B(1_G, \epsilon_2)$ implies

\begin{center}  $d(g_0^2(g_1^2(g)^2)^2, g_0^2(g_1^2)^2) \leq \frac{1}{9}$

$d(g_1^2(g)^2  , g_1^2) \leq \frac{\epsilon_1}{9}$
\end{center}

\noindent and for $g\in B(1_G, \epsilon_2) \setminus \ker(\phi)$ that $\phi(g) \notin F_1$.  Select $g_2$ so that $g_2, g_2^2\in B(1_G, \frac{\epsilon_2}{3})\setminus \ker(\phi)$.  Let $F_2$ be a finitely generated basic subgroup of $H$ containing $\langle F_1\cup \{\phi(g_2)\}\rangle$.  Supposing we have selected $g_0, \ldots, g_n$ and $\epsilon_1, \ldots, \epsilon_n$ and $F_0, \ldots, F_n$ in this way, we select $\epsilon_{n+1}>0$ such that $g\in B(1_G, \epsilon_{n+1})$ implies

\begin{center}  $d(g_0^2(g_1^2(\cdots g_n^2(g)^2 \cdots)^2)^2, g_0^2(g_1^2(\cdots (g_n^2)^2 \cdots)^2)^2)\leq \frac{1}{3^{n+1}}$

$d(g_1^2(\cdots g_n^2(g)^2 \cdots)^2, g_1^2(\cdots g_n^2 \cdots)^2)\leq\frac{\epsilon_1}{3^{n+1}}$

$\vdots$

$d(g_n^2(g)^2 , g_n^2) \leq \frac{\epsilon_n}{3^{n+1}}$

\end{center}

\noindent and for $g\in B(1_G, \epsilon_{n+1})\setminus \ker(\phi)$ that $\phi(g) \notin F_n$.  Select $g_{n+1}$ so that $g_{n+1}, g_{n+1}^2\in B(1_G, \frac{\epsilon_{n+1}}{3})\setminus \ker(\phi)$.  Pick a finitely generated basic subgroup $F_{n+1}$ containing $\langle F_n \cup \{\phi(g_{n+1})\} \rangle$.  Notice that for each $n\in\omega$ the sequence $g_n^2(g_{n+1}^2(\cdots g_{m-1}^2g_m^2 \cdots )^2)^2$ is Cauchy and converges to an element $j_n$.  Moreover it is clear that for $n\geq 1$ we have $j_n\in B(1_G, \epsilon_n)$.  The relations $j_n = g_n^2j_{n+1}^2$ are clear by continuity of multiplication.

We let $h_n = \phi(g_n)$ and $y_n = \phi(j_n)$ for all $n\in \omega$.  Performing the same argument as before, we contradict Lemma \ref{basic}.

Finally, we suppose $\phi: G\rightarrow H$ has locally compact Hausdorff domain and $\aleph_1$-free codomain and for contradiciton suppose that $\ker(\phi)$ is not open.  We inductively define nesting sequences $\{U_n\}_{n\in \omega}$ and $\{V_n\}_{n\in \omega}$ of open neighborhoods of $1_G$ such that $U_0 \supseteq V_0 \supseteq U_1\supseteq V_1 \supseteq \cdots$ and $\overline{V_n} \subseteq U_n$,  as well as a sequence $\{g_n\}_{n\in \omega}$ of elements in $G$ and finitely generated basic subgroups $F_0 \subseteq \cdots$.  Let $U_0 = G$ and select a neighborhood $V_0$ of $1_G$ such that $\overline{V_0}$ is compact.  Select $g_0$ so that $g_0, g_0^2 \in V_0 \setminus \ker(\phi)$.  Select a finitely generated basic subgroup $F_0$ which includes $\langle \phi(g_0)\rangle$.  By Lemma \ref{nicelemma} select $U_1 \subseteq V_1$ such that $g \in U_1$ implies

\begin{center}  $g_0^2g^2 \in V_0$
\end{center}

\noindent and if $g\in U_1 \setminus \ker(\phi)$ we have $\phi(g) \notin F_0$.  Pick an open neighborhood $V_1$ of $1_G$ such that $\overline{V_1} \subseteq U_1$ and select $g_1$ such that  $g_1, g_1^2 \in V_1 \setminus \ker(\phi)$.  Let $F_1$ be a basic finitely generated group including $\langle F_0 \cup \{ \phi(g_1)\}\rangle$.

Suppose we have selected neighborhoods $U_0, \ldots, U_n$ and $V_0, \ldots, V_n$ as well as elements $g_0, \ldots, g_n$ and basic free groups $F_0, \ldots, F_n$ in this manner.  Select a neighborhood $U_{n+1}$ of $1_G$ such that $g\in U_{n+1}$ implies

\begin{center}  $g_0^2(g_1^2(\cdots g_n^2(g)^2 \cdots)^2)^2 \in V_0$

$g_1^2(g_2^2(\cdots g_n^2(g)^2 \cdots)^2)^2 \in V_1$

$\vdots$

$g_n^2(g)^2 \in V_n$
\end{center}

\noindent and if $g\in U_{n+1}\setminus \ker(\phi)$ we have $\phi(g) \notin F_n$.  Pick open neighborhood $V_{n+1}$ of $1_G$ such that $\overline{V_{n+1}} \subseteq U_{n+1}$ and select $g_{n+1}$ such that $g_{n+1}, g_{n+1}^2 \in V_{n+1} \setminus \ker(\phi)$.  Let $F_{n+1}$ be a basic finitely generated subgroup including $\langle F_n \cup \{\phi(g_{n+1})\} \rangle$.  Define compact sets $K_n$ for $n\in \omega$ by letting $K_n = g_0^2(g_1^2(\cdots  g_n^2(\overline{V_{n+1}})^2 \cdots)^2)^2$.  It is easy to see that $\overline{V_0} \supseteq K_0 \supseteq K_1 \supseteq \cdots$ and so we may select $j_0 \in \bigcap_{n\in \omega}K_n$.  For $n \geq 1$ select $j_n\in \overline{V_{n+1}}$ such that $j_0 = g_0^2(\cdots g_n^2(j_n)^2 \cdots)^2$.  Let $h_n = \phi(g_n)$ and $y_n = \phi(j_n)$.  Since $y_0 =  h_0^2(\cdots h_n^2(y_{n+1})^2 \cdots)^2$ for all $n \in\omega$ and $H$ is locally free we get relations $y_n = h_n^2y_{n+1}^2$.  We derive a contradiction by arguing in the same manner as for cm-slenderness.
\end{proof}
\end{section}

\begin{section}{Automatic continuity in the abelian case}\label{abeliancase}

To avoid confusion we continue using multiplicative group notation, unless otherwise stated, despite the fact that some groups under discussion will be abelian.  We give definitions (see \cite{Fu}):

\begin{definition}\label{algebraicallycompactdef}  An abelian group $H$ is \emph{algebraically compact} if $H$ is a direct summand of a Hausdorff compact abelian group.
\end{definition}

The algebraically compact groups are closed under inverse limits, and finite abelian groups are obviously algebraically compact.  For each prime $p$ we have an inverse system of abelian groups $\mathbb{Z}/p^{n+1}\mathbb{Z} \rightarrow \mathbb{Z}/p^n\mathbb{Z}$ and let $J_p$ denote the inverse limit (the \emph{$p$-adic completion of $\mathbb{Z}$}.)  We also have an inverse system of abelian groups $\mathbb{Z}/n_0\mathbb{Z} \rightarrow \mathbb{Z}/n_1\mathbb{Z}$ (here $n_1 \mid n_0$) and denote by $\hat{\mathbb{Z}}$ the inverse limit (the $\mathbb{Z}$-adic completion of $\mathbb{Z}$.)  Both $J_p$ and $\hat{\mathbb{Z}}$ are algebraically compact and $J_p$ carries a natural group topology under which it is homeomorphic to the Cantor set.

An element $a$ of $\hat{\mathbb{Z}}$ has a representation of form $a = (a_1 + 2!\mathbb{Z}, a_2 + 3!\mathbb{Z}, \ldots)$ which is formally represented by the sum $\sum_{n=1}^{\infty}n!a_n$.  Two formal sums $\sum_{n=1}^{\infty}n!a_n$ and $\sum_{n=1}^{\infty}n!b_n$ represent the same element in $\hat{\mathbb{Z}}$ provided for all $m \geq 1$ we have

\begin{center}  $(m+1)! \mid \sum_{n=1}^m n!a_n - \sum_{n=1}^mn!b_n$
\end{center}

\begin{definition}\label{cotorsiondef}  An abelian group $H$ is \emph{cotorsion} if it is the homomorphic image of an algebraically compact group.
\end{definition}

\begin{definition}\label{cotorsionfreedef}  An abelian group $H$ is \emph{cotorsion-free} if it does not contain a nontrivial cotorsion group.  Equivalently $H$ is cotorsion-free if $H$ does not contain torsion, $\mathbb{Q}$, or a copy of the $p$-adic integers $J_p$ for any prime $p$ \cite[Theorem 13.3.8]{Fu}.
\end{definition}

\begin{definition}\label{reduced}  A torsion-free abelian group is \emph{reduced} if it contains no copy of $\mathbb{Q}$.
\end{definition}

\begin{definition}\label{Ulmsubgroup}  The \emph{first Ulm subgroup} of an abelian group $H$ is the subgroup $U(H) = \bigcap_{n \geq 1} H^n = \{h\in H\mid (\forall n\geq 1)(\exists h_n) h =h^n\}$. 
\end{definition}

\begin{definition}\label{lineartopology}  A topology on an abelian group $H$ is \emph{linear} if there exists a filter $\mathcal{F}$ of subgroups of $H$ such whose elements form a basis for the open neighborhoods of $1_H$.
\end{definition}

\begin{definition}\label{slenderdef}  An abelian group $H$ is \emph{slender} if for every homomorphism $\phi: \prod_{\omega}\mathbb{Z} \rightarrow H$ there exists some $m\in \omega$ such that $\phi = \phi \circ p_m$ where $p_m: \prod_{\omega}\mathbb{Z} \rightarrow \bigoplus_{n=0}^m\mathbb{Z} \times (0)_{n=m+1}^{\infty}$ is the retraction which projects the first $m+1$ coordinates.  Equivalently $H$ is slender if $H$ does not contain torsion, $\mathbb{Q}$, $\prod_{\omega}\mathbb{Z}$ or a copy of the $p$-adic integers $J_p$ for any prime $p$ \cite[Theorem 13.3.5]{Fu}.  Equivalently $H$ is slender if $H$ is torsion-free, reduced and contains no subgroup that admits a complete non-discrete metrizable linear topology \cite[Theorem 13.3.1]{Fu}.
\end{definition}

We prove a lemma which follows along the lines of \cite[Theorem 3.1]{EdFi}:

\begin{lemma}\label{getoff}  If $\phi: G \rightarrow H$ has completely metrizable or locally compact Hausdorff domain and cotorsion-free abelian codomain then $\ker(\phi)$ is closed.
\end{lemma}

\begin{proof}  Suppose $\phi: G \rightarrow H$ is a homomorphism with completely metrizable domain and cotorsion-free codomain and let $d$ be a complete metric compatible with the topology on $G$.  Suppose for contradiction that $\ker(\phi)$ is not closed.  If $g\in \overline{\ker(\phi)}\setminus \ker(\phi)$ then for there every neighborhood $U$ of $1_G$ we have $\phi(g)\in \phi(U) \setminus \{1_H\}$.  Then letting $h = \phi(g)$ and $H_{\infty} = \bigcap_{n\in \omega} \phi(B(1_G, \frac{1}{n}))$ we get $h\in H_{\infty} \setminus\{1_H\}$ and $H_{\infty}$ is easily seen to be a subgroup.  We obtain a contradiction by finding a nontrivial homomorphic image of the algebraically compact $\hat{\mathbb{Z}}$ in $H$, and since a homomorphic image of an algebraically compact group is cotorsion we will be finished.

Since $H$ is torsion-free and reduced we have that the first Ulm subgroup $U(H)$ is trivial.  We show that for each sequence of integers $\{a_n\}_{n\in \omega \setminus \{0\}}$ there exists a $j\in G$ for which (under additive notation) $(m+1)! \mid \phi(j) - \sum_{n=1}^m n!a_i h$ for all $m\geq 1$.  Then because $U(H)$ is trivial we get a well-defined $\psi: \hat{\mathbb{Z}} \rightarrow H$ given by $\psi(\sum_{n = 1}^{\infty}n!a_i) = \phi(j)$.  Since $h\in \psi(\hat{\mathbb{Z}})\setminus \{1_H\}$ we will have our nontrivial homomorphism.

Let a sequence $\{a_n\}_{n\in \omega}$ be given.  Select $g_1\in G$ such that $\phi(g_1) = h^{a_1}$.  Pick a neighborhood $U_2$ of $1_G$ such that $g'\in U_2$ implies $d(g_1(g')^{2!}, g_1)\leq \frac{1}{2}$.  Select $g_2 \in U_2 \cap \phi^{-1}(h^{a_2})$ (this is possible since $\phi$ surjects $U_2$ onto $H_{\infty}$ and $h \in H_{\infty}$).  Supposing we have selected $g_1, \ldots, g_n$ and $U_2, \ldots, U_n$ we select a neighborhood $U_{n+1}$ of $1_G$ such that $g'\in U_{n+1}$ implies 

\begin{center}  $d(g_1(g_2(\cdots g_n(g')^{(n+1)!}\cdots)^{3!})^{2!}, g_1(g_2(\cdots  g_n \cdots)^{3!})^{2!})\leq \frac{1}{2^n}$

$\vdots$

$d(g_n(g')^{(n+1)!}, g_n)\leq \frac{1}{2^n}$
\end{center}

\noindent Select $g_{n+1}\in U_{n+1}\cap \phi^{-1}(h^{a_{n+1}})$.  Fixing a $q \geq 1$ it is clear that the sequence $g_q(g_{q+1}(\cdots g_n\cdots)^{(q+2)!})^{(q+1)!}$ is Cauchy and therefore converges to, say, $j_q$.  We have $j = j_1 = g_1(\cdots g_{n-1}(j_n)^{n!} \cdots)^{2!}$ for each $n\geq 1$ by continuity of multiplication.  The relationship $(m+1)! \mid\phi(j) - \sum_{n=1}^mn!a_nh$ is now clear for all $m$.

Suppose now that $G$ is locally compact Hausdorff and for contradiction that $h\in H_{\infty} = \bigcap_{U \in \mathcal{U}}\phi(U)$ is nontrivial where $\mathcal{U}$ denotes the collection of open neighborhoods of $1_G$.  We again show that each sequence $\{a_n\}_{n\geq 1}$ has an element.  Let $V = V_1$ be a neighborhood of $1_G$ for which $\overline{V}$ is compact.  Pick $g_1\in V \cap \phi^{-1}(h^{a_1})$.  Supposing we have selected sequences $g_1, \ldots g_n$ and open neighborhoods $V_1, \ldots, V_n$ of $1_G$ in this way we select a neighborhood $V_{n+1}$ of $1_G$ such that $g_nV_{n+1}^{(n+1)!}\subseteq V_n$.  Let $g_{n+1}\in V_{n+1}\cap \phi^{-1}(h^{a_{n+1}})$.  Let $K_n = g_1(\cdots  g_n(\overline{V_{n+1}})^{(n+1)!}\cdots)^{2!}$ we then select $j \in \bigcap_{n \geq 1}K_n$ and notice once again that $(m+1)! \mid\phi(j) - \sum_{n=1}^mn!a_nh$ for all $m\geq 1$.
\end{proof}

\begin{proof}(of Theorem \ref{characterizationoflcHslenderabelian})

(1)  Suppose an abelian group $H$ is cm-slender.  Then $H$ cannot contain torsion, for then $H$ would contain some cyclic group of prime order $\mathbb{Z}/p\mathbb{Z}$.  The group $\prod_{\omega}\mathbb{Z}/p\mathbb{Z}$ is compact metrizable in a natural way and any homomorphism from $\bigoplus_{\omega}\mathbb{Z}/p\mathbb{Z} \leq \prod_{\omega}\mathbb{Z}/p\mathbb{Z}$ to $\mathbb{Z}/p\mathbb{Z}$ extends to a homomorphism on the entirety of $\prod_{\omega}\mathbb{Z}/p\mathbb{Z}$ by a vector space argument, so that it is quite easy to construct a homomorphism from $\prod_{\omega}\mathbb{Z}/p\mathbb{Z}$ to $\mathbb{Z}/p\mathbb{Z} \leq H$ which does not have an open kernel.

Also, $H$ cannot have a copy of $\mathbb{Q}$ since otherwise there exists a homomorphism from $\mathbb{R}$ to $\mathbb{Q} \leq H$ which does not have open kernel.  Neither can $H$ have a copy of a group which admits a non-discrete Polish topology, since then the inclusion map would witness that $H$ is not cm-slender.

Supposing $H$ is a group which is torsion-free, reduced and contains no subgroup which admits a non-discrete Polish topology.  Since $J_p$ has a non-discrete metrizable compact topology, we know that $H$ cannot contain any $J_p$ and so $H$ is cotorsion-free.  Let $\phi: G \rightarrow H$ be a homomorphism with $G$ completely metrizable.  Since $H$ is cotorsion-free, we have by Lemma \ref{getoff} that $\ker(\phi)$ is closed.  Supposing for contradiction that $\ker(\phi)$ is not open, we get a sequence $\{g_n\}_{n\in \omega}$ of elements of $G$ which converges to $1_G$ and such that $g_n \notin \ker(\phi)$.  Letting $G_{\infty} \leq G$ be the smallest closed subgroup of $G$ containing the elements of $\{g_n\}_{n\in \omega}$, we have that $G_{\infty}$ is Polish.  Also, $\ker(\phi\upharpoonright G_{\infty}) = G_{\infty}\cap \ker(\phi)$ is closed in $G_{\infty}$, and by how we selected $\{g_n\}_{n\in \omega}$ we know $\ker(\phi\upharpoonright G_{\infty})$ is not open in $G_{\infty}$.  The group $G_{\infty}/\ker(\phi\upharpoonright G_{\infty})$ is again a Polish group \cite[2.3.iii]{Ke} and not discrete by considering the cosets $g_n\ker(\phi\upharpoonright G_{\infty})$.  The map $\phi$ descends to an injective homomorphism $\overline{\phi}: G_{\infty}/\ker(\phi)\rightarrow H$.  Then $H$ contains a subgroup which admits a non-discrete Polish topology and we have a contradiction.

(2)  Suppose an abelian group $H$ is lcH-slender.  Then $H$ cannot contain torsion, $\mathbb{Q}$ or any $J_p$ by the reasoning as in (1) and so $H$ is cotorsion-free.

Suppose on the other hand that $H$ is cotorsion-free.  Let $\phi: G \rightarrow H$ be a homomorphism with locally compact Hausdorff domain.  By Lemma \ref{getoff} we know $\ker(\phi)$ is closed.  Then $G/\ker(\phi)$ is a locally compact abelian Hausdorff group and $\phi$ passes to an injective homomorphism $\overline{\phi}: G/\ker(\phi) \rightarrow H$.  By \cite[Theorem 25]{Mo} there exists an open subgroup $U$ of $G/\ker(\phi)$ which is topologically isomorphic to $\mathbb{R}^n \times K$ where $K$ is a compact group.  We show $U$ is trivial, so that $G/\ker(\phi)$ is discrete and $\ker(\phi)$ is open.  First of all, the superscript $n$ must be $0$ since $\overline{\phi}$ is injective and $H$ cannot contain $\mathbb{Q}$ as a subgroup.  But it is clear that $K$ must be trivial as well since otherwise $\phi(K)$ would be nontrivial cotorsion.
\end{proof}

The proof of Theorem \ref{lcHslender} now follows easily.  If a group $H$ is $\aleph_1$-free abelian, then it cannot contain torsion or $\mathbb{Q}$.  Also $H$ cannot contain any $J_p$ since then it would also contain the additive group of $\mathbb{Z}[\frac{1}{q}]$ for every prime $q\neq p$, and hence contain a countable subgroup which is not free abelian.  Thus an $\aleph_1$-free abelian group is cotorsion-free and we apply Theorem \ref{characterizationoflcHslenderabelian}.

We provide some definitions towards Theorem \ref{strongly} (see \cite{Me}):

\begin{definition}\label{purity}  If $H$ is $\kappa$-free abelian we say a subgroup $M\leq H$ is \emph{$\kappa$-pure} if $M$ is a direct summand of $\langle M \cup X\rangle$ for each set $X \subseteq H$ of cardinality $<\kappa$.
\end{definition}

\begin{definition}\label{stronglydef}  A $\kappa$-free abelian group $H$ is \emph{strongly $\kappa$-free abelian} if every subset $X \subseteq H$ of cardinality $<\kappa$ is contained in a $\kappa$-pure subgroup of $H$ generated by fewer than $\kappa$ elements.
\end{definition}

\begin{proof}(of Theorem \ref{strongly})  Suppose $\phi: G \rightarrow H$ has completely metrizable domain and strongly $\aleph_1$-free codomain.  Let $d$ be a complete metric compatible with the topology of $G$.  Supposing that $\ker(\phi)$ is not open we select $g_n\in B(1_G, \frac{1}{n}) \setminus \ker(\phi)$.  Select a countable $\aleph_1$-pure subgroup $M \supseteq \{\phi(g_n)\}_{n\in \omega}$.  Fix a free abelian generating set for $M$ and let $L: M \rightarrow \omega$ be the length function.  Let $k_n = L(\phi(g_n))$ for each $n\in \omega$.  We define a subsequence $\{n_q\}_{q\in \omega}$ inductively.  Let $n_0 = 0$ and supposing we have defined $n_0, \ldots, n_q$ we let $n_{q+1}$ be such that

\begin{center}  $d(g_{n_0}(g_{n_1}(\ldots g_{n_q}(g_{n_{q+1}})^{k_{n_q}+1} \ldots)^{k_{n_1}+2})^{k_{n_0} +2}, g_{n_0}(g_{n_1}(\ldots g_{n_q} \ldots)^{k_{n_1}+2})^{k_{n_0} +2}) \leq \frac{1}{2^q}$

$d(g_{n_1}(g_{n_2}(\cdots g_{n_q}(g_{n_{q+1}})^{k_{n_q}+2} \cdots)^{k_{n_2}+2})^{k_{n_1}+2}, g_{n_1}(g_{n_2}(\cdots  g_{n_q}\cdots)^{k_{n_2}+2})^{k_{n_1}+2}) \leq \frac{1}{2^q}$

$\vdots$

$d(g_{n_q}(g_{n_{q+1}})^{k_{n_q}+2}, g_{n_q})\leq \frac{1}{2^q}$
\end{center}

\noindent For each $m\in \omega$ the sequence $g_{n_m}(g_{n_{m+1}}(\cdots g_{n_q} \cdots)^{k_{n_{m+1}}+2})^{k_{n_m}+2}$  is Cauchy and therefore converges to an element $j_m$.  Letting $\rho: \langle M \cup \{\phi(j_m)\}_{m\in \omega} \rangle \rightarrow M$ be a retraction, we derive a contradiction as before.

Since for abelian groups cm-slenderness implies n-slenderness and lcH-slenderness, we are done.  
\end{proof}

There is an alternative proof for the fact that strongly $\aleph_1$-free groups are n- and lcH-slender which uses infinitary logic.  If $H$ is strongly $\aleph_1$-free then $H$ has the same $L_{\infty \omega_1}$ theory as free abelian groups \cite{Ek}.  Free abelian groups are slender, and slenderness is $L_{\infty \omega_1}$ axiomatizable \cite{SKo}, so $H$ is slender.  Slender groups are n-slender \cite{Ed} and they are also lcH-slender by Theorem \ref{characterizationoflcHslenderabelian}, so we are done.
\end{section}


\begin{thebibliography}{abcdefghijk}

\bibitem[B]{B} R. Baer, \emph{Abelian groups without elements of finite order}, Duke Math. J. 3 (1937), 68-122.

\bibitem[CC]{CC} G. Conner, S. Corson, \textit{A note on automatic continuity}, (to appear in Proc. Amer. Math. Soc.) arXiv 1710.04787

\bibitem[D]{D} R. Dudley, \textit{Continuity of homomorphisms}, Duke Math. J. 28 (1961), 587-594.

\bibitem[Ed]{Ed} K. Eda, \textit{Free $\sigma$-products and noncommutatively slender groups}, J. Algebra 148 (1992), 243-263.

\bibitem[EdFi]{EdFi} K. Eda, H. Fischer, \textit{Cotorsion-free groups from a topological viewpoint}, Topol. Appl. 214 (2016) 21-34.

\bibitem[Ek]{Ek} P. Eklof, \emph{Infinitary equivalence of abelian groups}, Fund. Math. 81 (1974), 305-314.

\bibitem[EkMe]{EkMe} P. Eklof, A. Mekler, Almost Free Modules: set theoretic methods, North-Holland (1990).

\bibitem[Fu]{Fu} L. Fuchs, Abelian Groups, Springer (2015).

\bibitem[H1]{H1} G. Higman, \emph{Almost free groups}, Proc. London Math. Soc. 1 (1951) 284-290.

\bibitem[H2]{H2} G. Higman,  \emph{Unrestricted free products and varieties of topological groups},  J. London Math. Soc. 27 (1952), 73-81.

\bibitem[Ke]{Ke} A. Kechris, \emph{Topology and descriptive set theory}, Topol. Appl. 58 (1994) 195-222.

\bibitem[Kh]{Kh} A. Khelif, \emph{Uncountable homomorphic images of Polish groups are not $\aleph_1$-free groups}, Bull. London Math. Soc. 37 (2005) 54-60.

\bibitem[MaS]{MaS} M. Magidor, S. Shelah, \emph{When does almost free imply free? (For groups, transversals, etc.)}, J. Amer. Math. Soc. 7 (1994), 769-830.

\bibitem[Me]{Me} A. Mekler, \emph{How to construct almost free groups}, Can. J. Math. 32 (1980), 1206-1228.

\bibitem[Mo]{Mo} S. Morris, Pontryagin Duality and the Structure of Locally Compact Abelian Groups, London Math. Soc. Lecture Notes 29, Cambridge U. Press, 1977.

\bibitem[S]{S} S. Shelah, \emph{A compactness theorem for singular cardinals, free algebra, Whitehead problem and transversals}, Israel J. Math. 21 (1975) 319-349.

\bibitem[SKo]{SKo} S. Shelah, O. Kolman, \emph{Infinitary axiomatizability of slender and cotorsion-free groups}, Bull. Belg. Math. Soc. Simon Stevin, 7 (2000), 623-629.
\end{thebibliography}
\end{document}